\documentclass[10pt]{article}
\usepackage[a4paper]{geometry}
\usepackage{amsthm}
\usepackage{amsmath}
\usepackage{amssymb}
\usepackage{amsfonts}
\usepackage{epsfig}
\usepackage[bf,SL,BF]{subfigure}
\usepackage[all]{xypic}

\begin{document}

\theoremstyle{theorem}
\newtheorem{Lemma}{Lemma}[section]
\newtheorem{Proposition}{Proposition}[section]
\newtheorem{Corollary}{Corollary}[section]
\newtheorem{Theorem}{Theorem}[section]
\newtheorem{Condition}{Condition}[section]
\newtheorem{Integrator}{Integrator}

\theoremstyle{definition}
\newtheorem{Definition}{Definition}[section]
\newtheorem{Remark}{Remark}[section]
\newtheorem{Example}{Example}[section]

\title{Structure preserving Stochastic Impulse Methods for stiff Langevin systems with a uniform global error of order 1 or 1/2 on position}

\author{Molei Tao\footnotemark[2] \footnotemark[4], Houman Owhadi\footnotemark[2] \footnotemark[3], Jerrold E. Marsden\footnotemark[2] \footnotemark[3]}

\renewcommand{\thefootnote}{\fnsymbol{footnote}}
\footnotetext[2]{Control \& Dynamical Systems, MC 107-81,}
\footnotetext[3]{Applied \& Computational Mathematics, MC 217-50, California Institute of Technology, Pasadena, CA 91125, USA}
\footnotetext[4]{Corresponding author; Email: mtao@caltech.edu}
\renewcommand{\thefootnote}{\arabic{footnote}}

\maketitle

\begin{abstract}
Impulse methods are generalized to a family of integrators for Langevin systems with quadratic stiff potentials and arbitrary soft potentials.
Uniform error bounds (independent from stiff parameters) are obtained on integrated positions allowing for coarse integration steps. The resulting integrators are explicit and structure preserving (quasi-symplectic for Langevin systems).
\end{abstract}

\section{Introduction}
\paragraph{Results:} This paper generalizes the impulse methods for stiff Hamiltonian systems \cite{Grubmuller:91,Tuckerman:92} to stiff stochastic Langevin systems. In the stochastic setting these integrators are quasi-symplectic as defined in \cite{MiTr2003}. Unform error bounds are obtained for both stochastic and deterministic settings.

More precisely, this paper is concerned with the numerical integration the following stiff SDEs:
\[
    \left\{ \begin {array} {rcl}
    M dq &=& p dt \\
    dp &=& -\nabla V(q)dt - \epsilon^{-1} K q dt - cpdt + \sigma dW
    \end {array} \right.
\]
which describes a stochastic mechanical system with a potential being sum of slow $V(q)$ and fast $\frac{1}{2}\epsilon^{-1}q^T Kq$, and the momentum being perturbed by noise and attenuated by friction.

When noise and friction are both present (i.e. $c\neq 0$, $\sigma\neq 0$), a 1st-order member of the proposed Stochastic Impulse Methods (SIMs) family will integrate position $q$ with a global error uniformly bounded by $CH^{1/2}$, where $H$ is the integration timestep and $C$ is a constant independent from $\epsilon^{-1}$ (provided that the solution remains bounded). The integrator is also shown to be quasi-symplectic. When noise and friction are absent (i.e. $c=0$, $\sigma=0$), the deterministic (and symplectic) version of the 1st-order SIM gives a uniform 1st-order global error on $q$, if again the solution is bounded. The error bound on momentum $p$, however, is not uniform here. Recall that increased accuracy and stability
was one of motivations supporting the development of mollified impulse methods \cite{Skeel:99,Sanz-Serna:08}.

\emph{Dynamical systems with multiple time scales pose a major problem in simulations because the small time steps required for stable integration of the fast motions lead to large numbers of time steps required for the observation of slow degrees of freedom} \cite{Hairer:04}. As seen from the error bounds, in the case of quadratic fast potential, SIMs provide a possibility of accurate integration with a choice of timestep not restricted by the stiffness $\epsilon^{-1}$, as long as position is the quantity of interest. In these cases, a large timestep can be adopted.

Also, SIMs are symplectic \cite{Hairer:04} and in fact variational \cite{MaWe:01} in the case of no noise no friction, and are quasi-symplectic \cite{MiTr2003} in the case of full Langevin. As a result of the preservation of structure, properties such as near preservation of energy or of associated Boltzmann-Gibbs invariant measure, as well as conservation of momentum maps could be obtained, which significantly benefit long time numerical integrations.

\paragraph{Related work:}
Many elegant methods have been proposed in the area of stiff Hamiltonian/Langevin integration with different focus and perspective.

Impulse methods, as well as other members of the exponential integrator family \cite{ExponentialIntegrators}, including Mollified Impulse Methods \cite{Skeel:99,Sanz-Serna:08} and Gautschi-type integrators \cite{Gautschi} are prevailing symplectic integrators for stiff Hamiltonian systems. They are however not directly extendable to stiff Langevin systems if integration with a large step is desired.  The general GLA \cite{BoOw:09} approach (see also \cite{SkIz02} for an extension of impulse methods  non-stiff Langevin systems) of constructing Langevin integrator from a symplectic scheme  by composing an Ornstein-Uhlenbeck flow with the symplectic integrator will not yield a uniform error bound in the case of stiff Langevins.  It is worth mentioning that impulse methods are not limited to quadratic stiff potentials (provided that the flow of the stiff part of the Hamiltonian is given).

The  implicit method approach for integrating stiff equations include the  LIN algorithm \cite{ZhSc:93} for stiff Langevin systems. However, it has been observed that \emph{Implicit methods in general fail to capture the effective dynamics of the slow time scale because they cannot correctly capture non-Dirac invariant distributions} \cite{LiAbE:08}. Moreover implicit methods are generally slower than explicit methods, provided they use comparable timesteps.

Implicit and explicit approaches were combined in a variational integration framework by defining the discrete Lagrangian via trapezoidal approximation of the soft potential and midpoint approximation of the stiff potential. The resulting IMEX for stiff Hamiltonian systems \cite{Stern:09} is explicit in the case of quadratic fast potential. Similar as the case of impulse methods, there is no easy way to extend IMEX to stiff Langevin systems.

The Hamilton-Jacobi derived homogenization method for multiscale Hamiltonian systems \cite{LeBris:07} enables the usage of a large timestep
for deterministic systems but can not directly be extended to stiff Langevin systems using the GLA approach \cite{BoOw:09}.

Multiscale methods that integrate the slow dynamics by averaging the effective contribution of the fast dynamics have been applied
to stiff Langevin systems. These
include Heterogeneous Multiscale Methods (HMM) \cite{MR2314852, MR2164093, MR2069938, CaSer08, Ariel:08}, equation free methods \cite{MR2041455, GiKeKup06, KevGio09}, and FLow AVeraging integratORS (FLAVORS) \cite{FLAVOR09}. We observe that these methods use mesoscopic timesteps, which are (usually) one or two orders of magnitude smaller than the large steps employed in SIMs. These methods also assume a separation of timescales, and therefore will not work for generic stiff Langevin systems that are not necessarily multiscale. In addition, based on averaging instantaneous drifts, both Heterogeneous Multiscale Methods and equation free methods (in their original form) require an identification of slow variables in general nonlinear cases (with exceptions such as in \cite{Seamless09}). Reversible and symmetric methods in these frameworks have been proposed \cite{MR1843642, Ariel:09, SerArTs09}. FLAVORS  are based on averaging instantaneous flows and do not require explicit identification of slow variables, and are symplectic (quasi-symplectic).

\section{Stochastic Impulse Methods}
Consider numerical integration of the following multiscale Langevin SDEs
\begin{equation}
    \left\{ \begin {array} {rcl}
    M dq &=& p dt \\
    dp &=& -\nabla V(q)dt - \epsilon^{-1} K q dt - cpdt + \sigma dW
    \end {array} \right.
    \label{dynamics}
\end{equation}
where $0 < \epsilon\ll 1$, $q \in \mathbb{R}^d$, $p \in \mathbb{R}^d$, $K$ is positive definite $d\times d$ matrix, $c$ and $\sigma$ are positive semi-definite $d\times d$ matrices, respectively indicating viscous damping coefficients and amplitudes of noises. We restrict ourselves to Euclidean phase spaces, although the method is readily generalizable to manifolds. In addition, we require that matrices $K$ and $c$ commute; a special case satisfying this requirement is $c$ being a scalar.

In the case of no noise no friction ($c=0$ and $\sigma=0$), the system degenerates to a deterministic mechanical system with Hamiltonian $H(q,p)=\frac{1}{2}p^T M^{-1} p+V(q)+\epsilon^{-1} \frac{1}{2} q^T K q$.

Also, the method as well as the uniform convergence theorem works for a more general open system:
\begin{equation}
    \left\{ \begin {array} {rcl}
    M dq &=& p dt \\
    dp &=& F(q)dt - \epsilon^{-1} K q dt - cpdt + \sigma dW
    \end {array} \right.
    \label{dynamics}
\end{equation}
but we stick to \eqref{dynamics} for simplicity in descriptions.

\smallskip

Denote by $\phi^f(\tau):\left(q^f(t),p^f(t)\right) \mapsto \left(q^f(t+\tau),p^f(t+\tau)\right)$ and $\phi^s(\tau):\left(q^s(t),p^s(t)\right) \mapsto \left(q^s(t+\tau),p^s(t+\tau)\right)$ respectively the $\tau$-flow maps of the autonomous SDE systems
\begin{equation}
    \left\{ \begin {array} {rcl}
        M dq^f &=& p^f dt \\
        dp^f   &=& - \epsilon^{-1} Kq^f dt - cp^f dt + \sigma dW
    \end {array} \right.
    \label{system1}
\end{equation}
and
\begin{equation}
    \left\{ \begin {array} {rcl}
        M dq^s &=& 0 \\
        dp^s   &=& -\nabla V(q^s) dt
    \end {array} \right.
    \label{system2}
\end{equation}
Since the first system is a linear SDE and the second is a free drift, flows of both can be obtained exactly.

Then Stochastic Impulse Methods(SIMs) are defined via compositions of $\phi^f$ and $\phi^s$. Here are several examples of SIMs with a timestep $H$:

\begin{Integrator}
\textbf{1st order SIM in the $c=0$, $\sigma=0$ case}, is given by the one step update of $\phi^s(H)\circ\phi^f(H)$:
\begin{eqnarray*}
    &\left\{ \begin{array}{rcl}
        q_{k'} &=& A_{11}(H)q_k+A_{12}(H)p_k \\
        p_{k'} &=& A_{21}(H)q_k+A_{22}(H)p_k \\
        q_{k+1} &=& q_{k'} \\
        p_{k+1} &=& p_{k'}-H \nabla V(q_{k'}) \\
    \end{array} \right. \\
    &\text{where }
        \left[ \begin{array}{cc}
        A_{11}(H) & A_{12}(H) \\
        A_{21}(H) & A_{22}(H)
        \end{array} \right]
        = \exp
        \left[ \begin{array}{cc}
        0 & M^{-1}H \\
        -\epsilon^{-1}KH & 0
        \end{array} \right] , \\
        &\begin{cases} q_0 = q(0) \\
            p_0 = p(0) \end{cases}
\end{eqnarray*}
\label{1stSIM_Hamiltonian}
\end{Integrator}

\begin{Remark}
The other 1st order SIM, as the above's dual, can be obtained via the one step update $\phi^f(H)\circ\phi^s(H)$. Both these 1st order composition schemes are well known as the Lie-Trotter splitting \cite{Trotter:59}.
\end{Remark}

\begin{Integrator}
\textbf{1st order SIM in the full Langevin case}, given by the same one step update $\phi^s(H)\circ\phi^f(H)$:
\begin{eqnarray*}
    &\left\{ \begin{array}{rcl}
        q_{k'} &=& B_{11}(H)q_k+B_{12}(H)p_k+Rq_k(H) \\
        p_{k'} &=& B_{21}(H)q_k+B_{22}(H)p_k+Rp_k(H) \\
        q_{k+1} &=& q_{k'} \\
        p_{k+1} &=& p_{k'}-H \nabla V(q_{k'}) \\
        \begin{bmatrix} Rq_k(H) \\ Rp_k(H) \end{bmatrix} &\sim& \mathcal{N}(\begin{bmatrix} 0 \\ 0 \end{bmatrix},\begin{bmatrix} \Sigma^2_{11}(H)) & \Sigma^2_{12}(H)) \\ \Sigma^2_{21}(H)) & \Sigma^2_{22}(H)) \end{bmatrix}),\text{i.i.d. normal distributed}
    \end{array} \right. \\
    &\text{where }
        \left[ \begin{array}{cc}
        B_{11}(H) & B_{12}(H) \\
        B_{21}(H) & B_{22}(H)
        \end{array} \right]
        = \exp
        \left[ \begin{array}{cc}
        0 & M^{-1}H \\
        -\epsilon^{-1}KH & -cH
        \end{array} \right] , \\
        &\begin{cases} q_0 = q(0) \\
            p_0 = p(0) \end{cases} , \\
        &\begin{cases}
        \Sigma^2_{11}(H) = \int_{s=0}^H \left( B_{12}(H-s) \sigma \sigma^T B_{12}^T(H-s)\right) ds \\
        \Sigma^2_{12}(H) = \int_{s=0}^H \left( B_{12}(H-s) \sigma \sigma^T B_{22}^T(H-s)\right) ds \\
        \Sigma^2_{21}(H) = \int_{s=0}^H \left( B_{22}(H-s) \sigma \sigma^T B_{12}^T(H-s)\right) ds \\
        \Sigma^2_{22}(H) = \int_{s=0}^H \left( B_{22}(H-s) \sigma \sigma^T B_{22}^T(H-s)\right) ds
        \end{cases}
\end{eqnarray*}
\label{1stSIM_Langevin}
\end{Integrator}

\begin{Remark}
    $\begin{bmatrix} Rq_k(H) \\ Rp_k(H) \end{bmatrix}$ indicates the value of $\int_{s=0}^{H} B(H-s) \begin{bmatrix} 0 \\ \sigma dW_s \end{bmatrix}$ and hence is a vectorial normal random variable with zero mean and covariance of $\begin{bmatrix} \Sigma_{11}^2(H) & \Sigma_{12}^2(H) \\ \Sigma_{21}^2(H) & \Sigma_{22}^2(H) \end{bmatrix}$.
\end{Remark}

\begin{Integrator}
\textbf{2nd order SIM in the full Langevin case}, given by the one step update $\phi^s(H/2)\circ\phi^f(H)\circ\phi^s(H/2)$:
\begin{eqnarray*}
    \left\{ \begin{array}{rcl}
        q_{k'} &=& q_k \\
        p_{k'} &=& p_k-\frac{H}{2}\nabla V(q_k) \\
        q_{k''} &=& B_{11}(H)q_{k'}+B_{12}(H)p_{k'}+Rq_k(H) \\
        p_{k''} &=& B_{21}(H)q_{k'}+B_{22}(H)p_{k'}+Rp_k(H) \\
        q_{k+1} &=& q_{k''} \\
        p_{k+1} &=& p_{k''}-\frac{H}{2}\nabla V(q_{k''}) \\
    \end{array} \right.
\end{eqnarray*}
\end{Integrator}

\begin{Remark}
This uses the 2nd order composition scheme known as the Strang or Marchuk splitting \cite{Strang:68, Marchuk:68}. When no noise or friction, i.e. $c=0$ and $\Sigma=0$, the resulting integrator degenerates to the prevailing Verlet-I/r-RESPA impulse method \cite{Grubmuller:91,Tuckerman:92}.
\end{Remark}

\begin{Remark}
Higher order SIMs can be obtained systematically since generic way for constructing higher order splitting/composition schemes exists \cite{Hairer:04}. For instance a 4th order SIM is given by $\phi^s(cH/2) \circ \phi^f(cH) \circ \phi^s((1-c)H/2) \circ \phi^f((1-2c)H) \circ \phi^s((1-c)H/2) \circ \phi^f(cH) \circ \phi^s(cH/2)$ where $c=\frac{1}{2-2^{1/3}}$ \cite{Neri:88}.
\end{Remark}

\section{Properties}
\subsection{Symplecticity}
In the case of $c=0$ and $\sigma=0$, since $\phi^s$ and $\phi^f$ are the exact flows of Hamiltonian systems, they are symplectic. Therefore SIMs, as compositions of the two, are symplectic.

In fact, SIMs here are not only symplectic but variational, in the sense that their equations of motion are obtained as critical point of a globally defined action, which is the integral of a discrete Lagrangian. Since SIMs are based on splitting schemes, and the original system is split to two Hamiltonian systems, backward error analysis can be done via Poisson brackets \cite{Hairer:04}, resulting in a global non-degenerate Hamiltonian that is exactly preserved. Then Legendre transformation gives the discrete Lagrangian and hence the variational structure.

When noise and friction are present, SIMs are quasi-symplectic for RL1 and RL2 in \cite{MiTr2003} can be easily checked to be true, i.e. they degenerate to symplectic methods if friction is set equal to zero and the Jacobian of the flow map is independent of $(q,p)$.

If in addition $c$ is isotropic, then SIMs are conformally symplectic, i.e. they preserve the precise symplectic area change associated to the flow of inertial Langevin processes \cite{McPe2001}.

\subsection{Uniform Convergence}
In the case of $c=0$ and $\sigma=0$, convergence of SIMs is guaranteed by the general construction of splitting schemes. In the full Langevin setting, analogous convergence results for the same splitting schemes can be easily obtained using generators of SDEs. By this approach, however, the error bound will contain the scaling factor $\epsilon^{-1}$ and therefore restrain the timestep from being large. We instead seek for uniform convergence results, i.e. error bounds that don't depend on $\omega$. It turns out such a uniform bound holds only for the position $q$ but not the momentum $p$.

\begin{Condition}
    We will prove a uniform bound on the scaled energy norm of the global error of Integrator \ref{1stSIM_Langevin} if the following conditions hold:
    \begin{enumerate}
        \item
            Matrices $c$ and $K$ commute. A special case could be $c$ being a scalar.
        \item
            $\lim_{\epsilon\rightarrow 0} \sqrt{\epsilon}\| c \|_2 \leq C$ for some constant $C$ independent of $\epsilon$, i.e. $c \leq O(\epsilon^{-1/2})$.
        \item
            $\sigma$ is independent of $\epsilon^{-1}$, in the sense that $\lim_{\epsilon\rightarrow 0} \epsilon^p \| \sigma \|_2 = 0$ for any $p>0$.
        \item
            In the integration domain of interest $\nabla V(\cdot)$ is bounded and Lipschitz continuous with coefficient $L$, i.e. $\| \nabla V(a)-\nabla V(b) \|_2 \leq L \| a-b \|_2$.
        \item
            Denote by $x(T)=(q(T),p(T))$ the exact solution to \eqref{dynamics}, and $x_T=(q_T,p_T)$ the discrete numerical trajectory given by Integrator \ref{1stSIM_Langevin}, then $\mathbb{E}\|x(T)\|_2^2 \leq C$ and $\mathbb{E}\|x_T\|_2^2 \leq C$ for some constant $C$ independent of $\epsilon^{-1}$ but dependent on initial condition $\mathbb{E}\| \begin{bmatrix} q_0 \\ p_0 \end{bmatrix} \|_2^2$, amplitude of noise $\sigma$ and friction $c$.

            Note that this condition usually holds due to preservation of Boltzmann-Gibbs invariant measure, whose parameter of temperature doesn't depend on $\epsilon^{-1}$ since noise is weak, and whose energy function is usually dominated by the positive definite fast potential (implying closed energy level sets).
    \end{enumerate}
    \label{Condition1}
\end{Condition}

\begin{Theorem}
    If Condition \ref{Condition1} holds, the $1^{st}$ order SIM (Integrator \ref{1stSIM_Langevin}) for multiscale Langevin system \eqref{dynamics} ($c\neq 0$, $\sigma\neq 0$) has in mean square sense a uniform global error of $O(H^{1/2})$ in $q$ and a non-uniform global error of $\epsilon^{-1/2}O(H^{1/2})$ in $p$, given a fixed total simulation time $T=NH$:
    \begin{eqnarray}
        (\mathbb{E} \| q(T)-q_T \|_2^2)^{1/2} &\leq& C H^{1/2} \\
        (\mathbb{E} \| p(T)-p_T \|_2^2)^{1/2} &\leq& \epsilon^{-1/2} \| \sqrt{K} \|_2 C H^{1/2}
    \end{eqnarray}
    where $q(T),p(T)$ is the exact solution and $q_T,p_T$ is the numerical solution; $C$ is a positive constant independent of $\epsilon^{-1}$ but dependent on simulation time $T$, scaleless elasticity matrix $K$, scaled damping coefficient $\sqrt{\epsilon}c$ ($O(1)$), amplitude of noise $\sigma$, slow potential energy $V(\cdot)$, and initial condition $\mathbb{E}\| \begin{bmatrix} q_0 \\ p_0 \end{bmatrix} \|_2^2$.
\label{UniformConvergence}
\end{Theorem}

\begin{proof}
We refer to the appendix for the proof.
\end{proof}

\begin{Remark}
By looking at the proof, one can be assured that all convergence results of SIMs apply to situations where the deterministic system is in a more general form of $M \frac{d^2}{dt^2}q=-\epsilon^{-1}Kq+F(q)$, where $F(q)$ doesn't have to be $-\nabla V(q)$.
\end{Remark}

In the special case of Hamiltonian system, the same integrator gains 1/2 more order of accuracies.

\begin{Condition}
    We will prove a uniform bound on the scaled energy norm of the global error of Integrator \ref{1stSIM_Hamiltonian} if the following conditions hold:
    \begin{enumerate}
        \item
            In the integration domain of interest $\nabla V(\cdot)$ is bounded and Lipschitz continuous with coefficient $L$, i.e. $\| \nabla V(a)-\nabla V(b) \|_2 \leq L \| a-b \|_2$.
        \item
            Denote by $x(T)=(q(T),p(T))$ the exact solution to \eqref{dynamics} with $c=0$ and $\sigma=0$, and $x_T=(q_T,p_T)$ the discrete numerical trajectory given by Integrator \ref{1stSIM_Hamiltonian}, then $\|x(T)\|_2^2 \leq C$ and $\|x_T\|_2^2 \leq C$ for some constant $C$ independent of $\epsilon^{-1}$ but dependent on initial condition $\| \begin{bmatrix} q_0 \\ p_0 \end{bmatrix} \|_2^2$.

            Note that this condition usually holds due to preservation of energy, which is usually dominated by the positive definite fast potential (implying closed energy level sets).
    \end{enumerate}
    \label{Condition2}
\end{Condition}

\begin{Theorem}
    If Condition \ref{Condition2} holds, the $1^{st}$ order SIM (Integrator \ref{1stSIM_Hamiltonian}) for multiscale Hamiltonian system (\eqref{dynamics} with $c=0$, $\sigma=0$) has a uniform global error of $O(H)$ in $q$ and a non-uniform global error of $\epsilon^{-1/2}$O(H) in $p$, given a fixed total simulation time $T=NH$:
    \begin{eqnarray}
        \| q(T)-q_T \|_2 &\leq& C H \\
        \| p(T)-p_T \|_2 &\leq& \epsilon^{-1/2} \| \sqrt{K} \|_2 C H
    \end{eqnarray}
    where $q(T),p(T)$ is the exact solution and $q_T,p_T$ is the numerical solution; $C$ is a positive constant independent of $\epsilon^{-1}$ but dependent on simulation time $T$, scaleless elasticity matrix $K$, slow potential energy $V(\cdot)$ and initial condition $\| \begin{bmatrix} q_0 \\ p_0 \end{bmatrix} \|_2$.
\end{Theorem}

\begin{proof}
It follows by simplifying the proof of Theorem \ref{UniformConvergence}.
\end{proof}

\subsection{Stability}
As one sees from Condition \ref{Condition1} and \ref{Condition2} (as another nonlinear demonstration of Lax equivalence theorem \cite{LaRi56}), stability is necessary for global convergence. Instability could either come from the problem itself (not all SDEs have bounded solutions in the mean square sense), or from imperfection in numerical integration schemes. Here consider the latter possibility only. It is shown that impulse methods are not unconditionally stable \cite{Skeel:99}, and its improvement, mollified impulse methods, are still susceptible to instability intervals (although narrower) in a linear example \cite{CalvoSena09}. Nevertheless, instability intervals of impulse method are already narrow regions; for instance, the first instability interval in the stiff example considered by \cite{CalvoSena09} is $0.544 < H < 0.553$. It is intuitive that instability intervals for the stochastic case with damping or higher order schemes will not be wider. Therefore one could still choose a large timestep $H$ in SIMs without hitting the instability, by at most a few integration tryouts with slightly varied $H$'s.

\section{Numerical Examples}
\subsection{2-spring systems with noise and friction}
\begin{figure} [ht]
\centering
\includegraphics[width=0.40\textwidth]{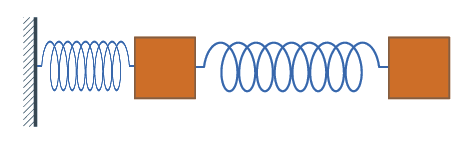}
\caption{\footnotesize 2-spring systems}
\label{example1}
\end{figure}

Consider a ``Wall -- linear stiff Spring -- Mass -- nonlinear soft Spring -- Mass'' system with both masses under isotropic noise and friction (Figure \ref{example1}). The governing equations write as:

\[ \begin{cases}
    dx &= p_x dt \\
    dy &= p_y dt \\
    dp_x &= -(\omega^2 x + (x-y)^3) dt - cp_x dt + \sigma dW^1_t \\
    dp_y &= -(y-x)^3 dt - cp_y dt + \sigma dW^2_t \\
\end{cases} \]

Note (1) this is a Langevin system with $H(x,y,p_x,p_y)=\frac{1}{2}p_x^2+\frac{1}{2}p_y^2+\frac{1}{2}\omega^2 x^2+\frac{1}{4}(y-x)^4$ (2) $y$ is a slow variable but $x$ is not purely fast (there is a slow component in it).

\begin{figure} [ht]
\centering
\subfigure[Full period case: $\sin(\omega H)=0$]{
\includegraphics[width=0.45\textwidth]{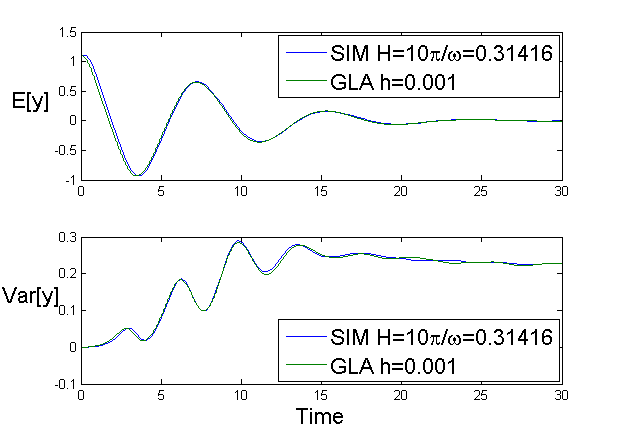}
\label{example1result1}
}
~
\subfigure[Quarter period case: $\cos(\omega H)=0$]{
\includegraphics[width=0.45\textwidth]{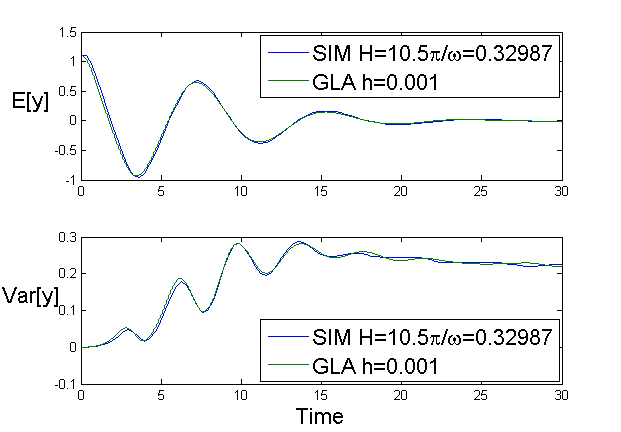}
\label{example1result2}
}
\caption{\footnotesize Empirical moments obtained by 1st-order SIM with large step $H$ and 1st-order GLA \cite{BoOw:09} with small step $h$. Parameters are $\omega=100$, $c=0.1$, $\beta=\frac{2c}{\sigma^2}=10$, $x(0)=0.8/\omega$, $y(0)=1.1+x(0)$, $p_x(0)=0$, $p_y(0)=0$; $h=0.1/\omega$ and $H$ is chosen to be not scaling with $\omega$ yet corresponding to a resonant frequency; empirical moments are obtained by averaging $5000$ simulations.}
\label{example1result}
\end{figure}

1st-order SIM (Integrator \ref{1stSIM_Langevin}) is compared in Figure \ref{example1result} to the benchmark of Geometric Langevin Integrator (GLA) \cite{BoOw:09} which is Boltzmann-Gibbs preserving and convergent. Agreements on empirical moments of integrated trajectories serve as evidences of structure preservation and convergence. The large timesteps used by SIM are chosen to be the resonance frequencies and they do produce stable accurate results. $O(\omega)$-fold acceleration is gained by SIM.

\subsection{Fermi-Pasta-Ulam problem}
\begin{figure} [h]
\includegraphics[width=\textwidth]{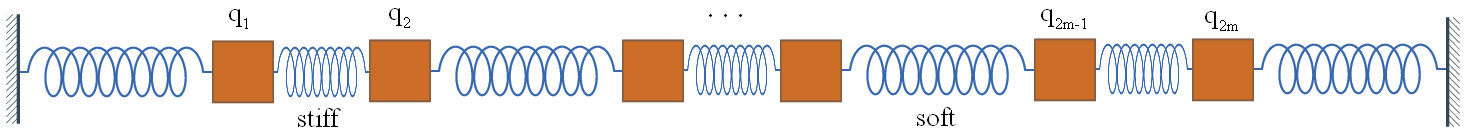}
\caption{\footnotesize Fermi-Pasta-Ulam problem \cite{FPU:55} -- 1D chain of alternatively connected harmonic stiff and non-harmonic soft springs}
\label{FPUfigure}
\end{figure}

Consider the deterministic Fermi-Pasta-Ulam (FPU) problem \cite{FPU:55} illustrated in Figure \ref{FPUfigure} and associated with the Hamiltonian
\begin{equation}
    H(q,p):=\frac{1}{2} \sum_{i=1}^m (p_{2i-1}^2+p_{2i}^2)+\frac{\omega^2}{4} \sum_{i=1}^m (q_{2i}-q_{2i-1})^2+ \sum_{i=0}^m (q_{2i+1}-q_{2i})^4
\end{equation}

Conventionally the following transformation is used
\begin{equation}
    \left\{ \begin{array}{rcl}
    x_i &=& (q_{2i}+q_{2i-1})/\sqrt{2} \\
    x_{m+i} &=& (q_{2i}-q_{2i-1})/\sqrt{2} \\
    y_i &=& (p_{2i}+p_{2i-1})/\sqrt{2} \\
    y_{m+i} &=& (p_{2i}-p_{2i-1})/\sqrt{2}
    \end{array} \right.
    i=1,...m
\end{equation}
so that the fast potential is diagonalized:
\[ \begin{cases}
    H(x,y) &= \frac{1}{2}\sum_{i=1}^{2m} y_i^2 + V_f(x) + V_s(x) \\
    V_f(x) &= \frac{\omega^2}{2}\sum_{i=1}^{m} x_{m+i}^2 \\
    V_s(x) &= \frac{1}{4}((x_1-x_{m+1})^4+\sum_{i=1}^{m-1}(x_{i+1}-x_{m+i+1}-x_i-x_{m+i})^4+(x_m+x_{2m})^4)
\end{cases} \]

\begin{figure} [ht]
\centering
\subfigure[1st-order SIM, large step $H=0.1$]{
\includegraphics[width=0.46\textwidth]{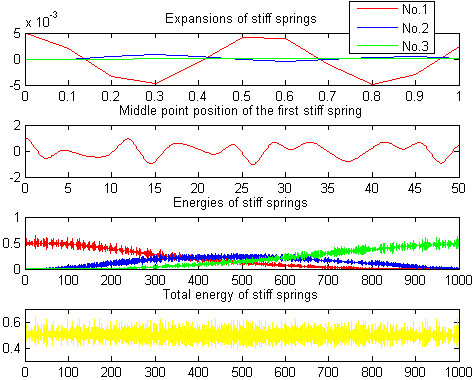}
\label{example2result1}
}
~
\subfigure[Variational Euler, small step $h=0.1/\omega=0.0005$]{
\includegraphics[width=0.46\textwidth]{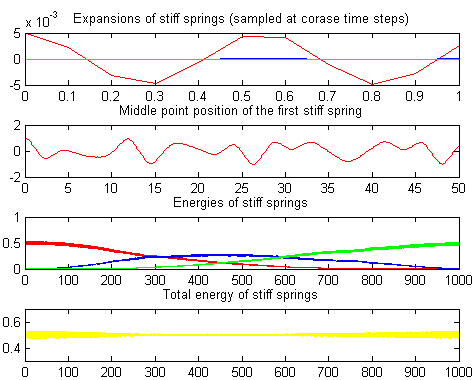}
\label{example2result2}
}
\caption{\footnotesize Simulations of FPU over $T=5\omega$. Parameters are $\omega=200$, $m=3$, $x(0)=[1,0, 0,1/\omega,0,0]$, $y(0)=[0,0,0,0,0,0]$. Different subplots use different time axes to accentuate different timescales: Subplot1 shows scaled expansions of three stiff springs $x_{m+i}$, which are fast variables; Subplot2 shows scaled middle point position of the first stiff spring $x_1$, which is one of the slow variables; Subplot3 shows the energy transferring pattern among stiff springs, which is even slower; Subplot4 shows the near-constant total energy of three stiff springs. The fast variables of stiff spring expansions are in fact oscillating much faster than shown in Subplots 1, for Subplots 1 are plotted by interpolating mesh points with a coarse mesh size of $H$.}
\label{example2result}
\end{figure}

The FPU problem is a well known benchmark problem \cite{FPUcomp07, Hairer:04} for multiscale integrators  because it exhibits different behaviors over widely separated timescales. The stiff springs (nearly) behave  like harmonic oscillator with period $\sim O(\omega^{-1})$. Then the centres of masses linked by stiff springs (i.e. the middle points of stiff springs) change over a timescale $O(1)$. The third timescale $O(\omega)$ is associated with the rate of energy exchange among stiff springs. On the other hand, in addition to conservation of energy, the total energy of stiff springs behave almost like a constant. Comprehensive surveys on FPU problem, including discussions on timescales and numerical recipes, can be found in \cite{Hairer:04, MR2275175}.

We present in Figure \ref{example2result} 1st-order SIM simulation (Integrator \ref{1stSIM_Hamiltonian}) together with variational Euler (a.k.a. symplectic Euler) simulation of FPU over a time span of $O(\omega)$. Good results are obtained by SIM beyond the timescale of $O(1)$ (as guaranteed by Theorem \ref{UniformConvergence}) but actually over $O(\omega)$, and 200-fold ($\omega=200$) acceleration is gained at the same time.

Notice that Mollified Impulse Methods with ShortAverage, LongAverage or LinearAverage filters \cite{Skeel:99} didn't accurately capture the rates of energy exchanging among stiff springs over $T=O(\omega)$ (results not shown).

\section{Acknowledgement}
This work is supported by NSF grant CMMI-092600. We thank J. M. Sanz-Serna for useful comments.

\section{Appendix}
\subsection{Proof of Theorem \ref{UniformConvergence}}

Throughout this subsection Condition \ref{Condition1} is assumed. For a concise writing we also abuse the notation $O(x^n)$, which indicates some entity whose norm $\leq Cx^n$, where $C$ is a constant that doesn't change with $\epsilon$, i.e. doesn't depend on $\epsilon^{-1}$.

\begin{Definition} \textbf{Scaled energy norm:}
    \begin{eqnarray*}
        &\Omega \triangleq \epsilon^{-1/2}\sqrt{K} \\
        &\| \begin{bmatrix} q \\ p \end{bmatrix} \|_E \triangleq \| \begin{bmatrix} q \\ \Omega^{-1}p \end{bmatrix} \|_2 = \sqrt{q^T q + \epsilon p^T K^{-1} p}
    \end{eqnarray*}
\end{Definition}
This is well defined because K is positive definite.

Since $\epsilon$ is very small, the following inequalities for converting between scaled energy norm and two-norm can be easily obtained:

\begin{Proposition}
    Let $x=\begin{bmatrix} q \\ p \end{bmatrix}$ be any vector, then
    \begin{eqnarray}
        \epsilon^{1/2} \| \sqrt{K} \|_2^{-1} \| x \|_2 = \| \Omega \|_2^{-1} \| x \|_2 \leq \| x \|_E \leq \| x \|_2 \\
        \| \begin{bmatrix} 0 \\ p \end{bmatrix} \|_E \leq \| \Omega^{-1} \|_2 \| \begin{bmatrix} 0 \\ p \end{bmatrix} \|_2 = \epsilon^{1/2} \| \sqrt{K}^{-1} \|_2 \| x \|_2
    \end{eqnarray}
    Also, vector-norm-induced matrix norms satisfy
    \begin{eqnarray}
        \| \begin{bmatrix} M_{11} & M_{12} \\ M_{21} & M_{22} \end{bmatrix} \|_E \triangleq \sup \frac{\| Mx \|_E}{\| x \|_E}
         =
        \| \begin{bmatrix} M_{11} & M_{12} \Omega \\ \Omega^{-1} M_{21} & \Omega^{-1} M_{22} \Omega \end{bmatrix} \|_2
    \end{eqnarray}
\end{Proposition}

\begin{Lemma}
    Let $B(s)=\begin{bmatrix} B_{11}(s) & B_{12}(s) \\ B_{21}(s) & B_{22}(s) \end{bmatrix}=\exp(s \begin{bmatrix} 0 & I \\ -\epsilon^{-1}K & c \end{bmatrix})$, and $Rq(s)$ be the $Rq_k(H)$ defined in Integrator \ref{1stSIM_Langevin} with $H=s$ and arbitrary $k$, then
    \begin{eqnarray}
        \| B_{11}(s) \|_2 &\leq& 1 \\
        \| B_{22}(s) \|_2 &\leq& 1 \\
        \| B_{12}(s) \|_2 &\leq& |s| \\
        \epsilon \| B_{21}(s) \|_2 &\leq& C_K|s| \\
        \epsilon^{1/2} \| B_{11}(s)-I \|_2 &\leq& C_K|s| \\
        \epsilon^{1/2} \| B_{22}(s)-I \|_2 &\leq& C_c|s| \\
        \mathbb{E} \| Rq(s) \|_2^2 &\leq& \frac{1}{3} \| \sigma \|_2^2 |s|^3 \\
        \epsilon^{1/2} \| B(s)-I \|_2 &\leq& C_{Kc}|s| \\
        \epsilon^{1/2} \| B(s)-I \|_E &\leq& C_{Kc}|s|
    \end{eqnarray}
    where $C_K$, $C_c$ and $C_{Kc}$ are some positive real constants (may indicate different values in different inequalities), respectively dependent on $K$, $\sqrt{\epsilon}c$, $K$ and $\sqrt{\epsilon}c$ but independent of $\epsilon^{-1}$.
    \label{ElementaryBound2}
\end{Lemma}

\begin{proof}
    Since $c$ and $K$ commute, they can be diagonalized simultaneously \cite{MatrixAnalysis}. By the theory of linear ordinary differential equations \cite{Perko}, one can hence diagonalize $B_{11}$, $B_{12}$, $B_{21}$, $B_{22}$ simultaneously. Since each diagonal element can be investigated individually, assume without loss of generality that $\Omega=[\omega]_{ij}=\epsilon^{-1/2}\sqrt{K}$ and $c$ are both scalars, and use the notation of scalar $\omega$ and scalar $c$ thereafter.

    Denote the damping ratio by $\zeta=\frac{c}{\omega}$. The solution to damped harmonic oscillator can be analytically obtained, and hence components of the flow operator $B_{11}$,$B_{12}$,$B_{21}$,$B_{22}$ as well.

    When $\zeta<1$ i.e. underdamping, which is usually the case since $\omega$ is large
    \begin{eqnarray}
        B_{11}(s)&=&e^{-\omega\zeta s} (cos(\omega\sqrt{1-\zeta^2}s)+\frac{\zeta }{\sqrt{1-\zeta^2}} sin(\sqrt{1-\zeta^2}s)) \\
        B_{12}(s)&=&\frac{e^{-\omega\zeta s}sin(\omega\sqrt{1-\zeta^2}s)}{\omega\sqrt{1-\zeta^2}}\\
        B_{21}(s)&=&-\omega\frac{e^{-\omega\zeta s}sin(\omega\sqrt{1-\zeta^2}s)}{\sqrt{1-\zeta^2}}\\
        B_{22}(s)&=&e^{-\omega\zeta s} (cos(\omega\sqrt{1-\zeta^2}s)-\frac{\zeta }{\sqrt{1-\zeta^2}} sin(\sqrt{1-\zeta^2}s))
    \end{eqnarray}

    When $\zeta=1$ i.e. critical damping,
    \begin{eqnarray}
        B_{11}(s)&=&e^{-\omega s}(1+\omega s) \\
        B_{12}(s)&=&e^{-\omega s}s \\
        B_{21}(s)&=&-\omega^2 e^{-\omega t}t \\
        B_{22}(s)&=&e^{-\omega s}(1-\omega s)
    \end{eqnarray}

    When $\zeta>1$ i.e. over damping,
    \begin{eqnarray}
        A(s)&\triangleq&e^{\omega s(-\zeta-\sqrt{\zeta^2-1})} \\
        B(s)&\triangleq&e^{\omega s(-\zeta+\sqrt{\zeta^2-1})} \\
        B_{11}(s)&=&\frac{\zeta(B-A)+\sqrt{\zeta^2-1}(A+B)}{2\sqrt{\zeta^2-1}} \\
        B_{12}(s)&=&\frac{-A+B}{2\omega\sqrt{\zeta^2-1}} \\
        B_{21}(s)&=&\frac{\omega(A-B)}{2\sqrt{\zeta^2-1}} \\
        B_{22}(s)&=&\frac{\zeta(A-B)+\sqrt{\zeta^2-1}(A+B)}{2\sqrt{\zeta^2-1}} \\
    \end{eqnarray}

    By routine investigations on local extremes using calculus, it can be shown in all three cases that

    \begin{eqnarray}
        \| B_{11}(s) \|_2 &\leq& 1 \\
        \| B_{22}(s) \|_2 &\leq& 1 \\
        \| B_{12}(s) \|_2 &\leq& s \\
        \| B_{21}(s) \|_2 &\leq& \omega^2 s \\
        \| B_{11}(s)-I \|_2 &\leq& \omega s \\
        \| B_{22}(s)-I \|_2 &\leq& \left\{ \begin {array} {ll} \omega s & \zeta \leq 1 \\ 2\zeta \omega s & \zeta > 1 \end{array} \right.
    \end{eqnarray}

    When $\zeta>1$, since $c=O(\epsilon^{-1/2})$ (Condition \ref{Condition1}), $2\zeta\omega s = O(\epsilon^{-1/2})s$. Therefore $\epsilon^{1/2} \| B_{22}(s)-I \|_2 \leq C_c|s|$ always holds.

    Also,
    \begin{eqnarray}
        \mathbb{E} \| Rq(s) \|_2^2 &=& \mathbb{E} \| \int_0^s B_{12}(t) \sigma dW_t \|_2^2 \nonumber \\
        &=& \int_0^s \| \sigma B_{12}(t) \|_2^2 dt \leq \frac{1}{3} \|\sigma\|_2^2 |s|^3
    \end{eqnarray}

    For a proof on norm bounds of the entire matrice we use only bounds of dimensionless block elements:
    \begin{eqnarray}
        \| B-I \|_2 &=& \| \begin{bmatrix} B_{11}-I & B_{12} \\ B_{21} & B_{22}-I \end{bmatrix} \|_2 \\
        &\leq& \| \begin{bmatrix} \Omega & 0 \\ 0 & \Omega \end{bmatrix} \|_2 \| \begin{bmatrix} \Omega^{-1}(B_{11}-I) & \Omega^{-1}B_{12} \\ \Omega^{-1}B_{21} & \Omega^{-1}(B_{22}-I) \end{bmatrix} \|_2 \\
        &=& \epsilon^{-1/2} \| \begin{bmatrix} O(s) & \epsilon^{1/2}O(s) \\ \epsilon^{-1/2}O(s) & O(s) \end{bmatrix} \|_2
    \end{eqnarray}

    It's easy to prove that for any scalar $a$
    \begin{equation}
        \| \begin{bmatrix} M_{11} & aM_{12} \\ M_{21} & M_{22} \end{bmatrix} \|_2 = \| \begin{bmatrix} M_{11} & M_{12} \\ aM_{21} & M_{22} \end{bmatrix} \|_2
    \end{equation}

    Therefore
    \begin{eqnarray}
        \| B-I \|_2 = \epsilon^{-1/2} \| \begin{bmatrix} O(s) & O(s) \\ O(s) & O(s) \end{bmatrix} \|_2 = \epsilon^{-1/2} O(s)
    \end{eqnarray}

    Similarly,
    \begin{eqnarray}
        \| B-I \|_E &=& \epsilon^{-1/2} \| \begin{bmatrix} B_{11}-I & B_{12}\Omega \\ \Omega^{-1}B_{21} & \Omega^{-1}B_{22}\Omega-I \end{bmatrix} \|_2 \nonumber\\
        &\leq& \epsilon^{-1/2} \| \begin{bmatrix} O(s) & \epsilon^{1/2}\epsilon^{-1/2}O(s) \\ \epsilon^{-1/2}\epsilon^{1/2}O(s) & \epsilon^{1/2}O(s)\epsilon^{-1/2} \end{bmatrix} \|_2 \nonumber\\
        &=& \epsilon^{-1/2} \| \begin{bmatrix} O(s) & O(s) \\ O(s) & O(s) \end{bmatrix} \|_2 \nonumber\\
        &=& \epsilon^{-1/2} O(s)
    \end{eqnarray}

\end{proof}

\begin{Remark}
    In the special case of $c=0$, bounds of block elements can be easily obtained since
    \begin{align*}
        |cos(\omega s)| &\leq 1 \\
        \epsilon K^{-1} |-\omega sin(\omega s)| &= | \frac{-\omega sin(\omega s)}{\omega^2} | \leq |s| \\
        \epsilon^{1/2}\sqrt{K}^{-1} | cos(\omega s)-1 | &= | -2sin^2(\omega s/2)/\omega| \leq | -2sin(\omega s/2)/\omega| \leq |s|
    \end{align*}
\end{Remark}

\begin{Lemma}
    The solution to SDE $dX=AXdt+f(X)dt+\Sigma dW_t$ can be written in the following integral form:
    \begin{equation}
        X(t)=e^{At}X(0)+\int_0^t e^{A(t-s)}f(X(s))ds+\int_0^t e^{A(t-s)} \Sigma dW_s
    \end{equation}
    \label{SDEsolution}
\end{Lemma}

\begin{proof}
    Let $Y(t)=e^{-At}X(t)$, then by Ito's formula and $dX=AXdt+f(X)dt+\Sigma dW_t$
    \begin{equation}
        dY=e^{-At}f(X(t))dt+e^{-At}\Sigma dW_t
    \end{equation}
    This in the integral form is
    \begin{equation}
        Y(t)=Y(0)+\int_0^t e^{-As}f(X(s))ds+\int_0^t e^{-As}\Sigma dW_s
    \end{equation}
    Hence
    \begin{equation}
        X(t)=e^{At}X(0)+\int_0^t e^{A(t-s)}f(X(s))ds+\int_0^t e^{A(t-s)}\Sigma dW_s
    \end{equation}
\end{proof}

\begin{Lemma}
    Consider two continuous stochastic dynamical systems, the original dynamics and the bridge dynamics:
    \begin{equation}
        \left\{ \begin{array}{rcl}
            dq&=& p dt \\
            dp&=&-\epsilon^{-1} Kq dt- \nabla V(q) dt -cpdt+\sigma dW_t\\
            q(0)&=&q_0 \\
            p(0)&=&p_0
        \end{array} \right. \\
        \label{OriginalDynamics2}
    \end{equation}
    \begin{equation}
        \left\{ \begin{array}{rcl}
            d\tilde{q}&=&\tilde{p} dt \\
            d\tilde{p}&=&-\epsilon^{-1} K\tilde{q} dt- \nabla V(q_0) dt -c\tilde{p}dt+\sigma dW_t\\
            \tilde{q}(0)&=&q_0 \\
            \tilde{p}(0)&=&p_0
        \end{array} \right. \\
        \label{BridgeDynamics2}
    \end{equation}
    Then $( \mathbb{E} \| \begin{bmatrix} \tilde{q}(H) \\ \tilde{p}(H) \end{bmatrix} - \begin{bmatrix} q(h) \\ p(h) \end{bmatrix} \|_E^2 )^{1/2} \leq C|H|^{3/2}$, where $C$ is a positive constant independent of $\epsilon^{-1}$ but dependent on the scaleless elasticity matrix $K$, scaled damping coefficient $\sqrt{\epsilon}c$, amplitude of noise $\sigma$, and slow potential $V(\cdot)$.
    \label{lemma1}
\end{Lemma}

\begin{proof}
    Rewrite the original dynamics \eqref{OriginalDynamics2} as
    \begin{equation}
        \left\{ \begin{array}{rcl}
            dq&=& p dt \\
            dp&=&\epsilon^{-1} Kq dt- \nabla V(q_0) dt + (\nabla V(q_0)-\nabla V(q)) dt -cpdt+\sigma dW_t\\
            q(0)&=&q_0 \\
            p(0)&=&p_0
        \end{array} \right. \\
    \end{equation}

    Let $x(t)=\begin{bmatrix} q(t) \\ p(t) \end{bmatrix}$, $\tilde{x}(t)=\begin{bmatrix} \tilde{q}(t) \\ \tilde{p}(t) \end{bmatrix}$, $B(t)=\exp(t \begin{bmatrix} 0 & I \\ -\epsilon^{-1}K & -c \end{bmatrix})$, $b=\begin{bmatrix} 0 \\ -\nabla V(q_0) \end{bmatrix}$, $g(q,p)=g(x)=\begin{bmatrix} 0 \\ \nabla V(q_0)-\nabla V(q) \end{bmatrix}$, and $\Sigma=\begin{bmatrix} 0 \\ \sigma \end{bmatrix}$. Then by Lemma \ref{SDEsolution} solutions to the original dynamics and bridge dynamics can be respectively written as:

    \begin{eqnarray}
        x(t) &=& B(t)x(0)+\int_0^t B(t-s)bds+\int_0^t B(t-s)\Sigma dW_s+\int_0^t B(t-s)g(x(s))ds \nonumber\\
        \tilde{x}(t) &=& B(t)x(0)+\int_0^t B(t-s)bds+\int_0^t B(t-s)\Sigma dW_s
    \end{eqnarray}

    Notice for any vector $y$ and positive $t$ that $\| B(t) y \|_E \leq \| y \|_E$,  because energy is decaying in the system $\ddot{q}+c\dot{q}+\epsilon^{-1}Kq=0$. Together with Cauchy-Schwarz we have

    \begin{eqnarray}
        \mathbb{E} \| \tilde{x}(t)-x(t) \|_E^2 &=& \mathbb{E} \| \int_0^t B(t-s)g(x(s)) ds \|_E^2 \nonumber\\
        &\leq& t \int_0^t \mathbb{E} \| B(t-s)g(x(s)) \|_E^2 ds \nonumber\\
        &\leq& t \int_0^t \mathbb{E} \| g(x(s)) \|_E^2 ds
        \label{avlsjoiuqiuqo}
    \end{eqnarray}

    By Condition \ref{Condition1}, assume $\nabla V(\cdot)$ is Lipschitz continuous with coefficient $L$, then almost surely
    \begin{eqnarray}
        \| g(x(s)) \|_E &=& \| \begin{bmatrix} 0 \\ \nabla V(q_0) - \nabla V(q(s)) \end{bmatrix} \|_E \nonumber\\
        &\leq& \sqrt{\epsilon} \| \sqrt{K}^{-1} \|_2 \| \nabla V(q_0) - \nabla V(q(s)) \|_2 \nonumber\\
        &\leq& L \sqrt{\epsilon} \| \sqrt{K}^{-1} \|_2 \| q(s)-q_0 \|_2
        \label{ayufcadsujdsai}
    \end{eqnarray}

    Similarly, since
    \begin{equation}
        x(t) = B(t)x(0)+\int_0^t B(t-s) \begin{bmatrix} 0 \\ -\nabla V(q(s)) \end{bmatrix} ds+\int_0^t B(t-s)\Sigma dW_s
    \end{equation}
    we have
    \begin{equation}
        \| x(s)-B(s)x_0-\int_0^s B(s-t)\Sigma dW_t \|_2 \leq \int_0^s \| \nabla V(q(t)) \|_2 dt
        \label{oqiurhoavijdapofosd}
    \end{equation}

    By Condition \ref{Condition1}, $\nabla V(\cdot)$ is bounded, and hence the above is $O(s)$.

    We now can bound \eqref{ayufcadsujdsai} and therefore \eqref{avlsjoiuqiuqo} with the aid of \eqref{oqiurhoavijdapofosd} and Lemma \ref{ElementaryBound2}:
    \begin{eqnarray}
        &~& \mathbb{E} \| q(s)-q_0 \|_2^2 \nonumber\\
        &\leq& \mathbb{E} \| x(s)-x_0 \|_2^2 \nonumber\\
        &\leq& \mathbb{E} \left( \| x(s)-B(s)x_0-\int_0^s B(t-s)\Sigma dW_s \|_2 + \| B(s)x_0-x_0 \|_2 + \| \int_0^s B(t-s)\Sigma dW_s \|_2 \right)^2 \nonumber\\
        &\leq& 3\mathbb{E} \left( \| x(s)-B(s)x_0-\int_0^s B(t-s)\Sigma dW_s \|_2^2 + \| B(s)x_0-x_0 \|_2^2 + \| \int_0^s B(t-s)\Sigma dW_s \|_2^2 \right) \nonumber\\
        &=& 3\left(O(s^2)+\epsilon^{-1} O(s^2) \mathbb{E} \| x_0 \|_2^2 + \int_0^s \sigma^2 (B_{12}(t-s)^2+B_{22}(t-s)^2)dt\right) \nonumber\\
        &=& O(s^2)+\epsilon^{-1} O(s^2) \mathbb{E} \| x_0 \|_2^2 + O(s^3) + O(s)
    \end{eqnarray}

    By Condition \ref{Condition1}, $\mathbb{E}\| x_0 \|_2^2=O(1)$. Therefore, the above expression is $\epsilon^{-1}O(s^2)+O(s)$.

    This gives $\mathbb{E} \| g(\tilde{x}(s)) \|_E^2 = O(s)$ independent of $\epsilon^{-1}$, and eventually $\mathbb{E} \| \tilde{x}(h)-x(h) \|_E^2  = O(h^3)$.
\end{proof}

\begin{Lemma}
    Consider the discrete stochastic dynamical system given by 1st-order SIM (Integrator \ref{1stSIM_Langevin}):
    \begin{equation}
        \left\{ \begin{array}{rcl}
            q_H&=&B_{11}(H)q_0+B_{12}(H)p_0+Rq(H) \\
            p_H&=&B_{21}(H)q_0+B_{22}(H)p_0+Rp(H)-H\nabla V(B_{11}(H)q_0+B_{12}(H)p_0+Rq(H)) \\
        \end{array} \right. \\
    \end{equation}
    Then a comparison with bridge dynamics \eqref{BridgeDynamics2} gives
    $\mathbb{E} \| q_H-\tilde{q}(H) \|_2^2 \leq CH^4$ and
    $\mathbb{E} \| \Omega^{-1} (p_H-\tilde{p}(H)) \|_2^2 \leq CH^4$, and therefore
    \begin{equation}
        (\mathbb{E} \| \begin{bmatrix} q_H \\ p_H \end{bmatrix} - \begin{bmatrix} \tilde{q}(H) \\ \tilde{p}(H) \end{bmatrix} \|_E^2)^{1/2} \leq CH^2
    \end{equation}
    where $C$'s are positive constants independent of $\epsilon^{-1}$ but dependent on scaleless elasticity matrix $K$, scaled damping coefficient $\sqrt{\epsilon}c$, amplitude of noise $\sigma$, and slow potential $V(\cdot)$.
    \label{lemma2}
\end{Lemma}

\begin{proof}
    The exact solution to the bridge dynamics is
    \begin{equation}
        \left\{ \begin{array}{rcl}
            \tilde{q}(H)&=&B_{11}(H)q_0+B_{12}(H)p_0+\int_0^H B_{12}(s)(-\nabla V(q_0))ds+Rq(H) \\
            \tilde{p}(H)&=&B_{21}(H)q_0+B_{22}(H)p_0+\int_0^H B_{22}(s)(-\nabla V(q_0))ds+Rp(H) \\
        \end{array} \right. \\
    \end{equation}

    Hence almost surely $\tilde{q}(H)-q_H=\int_0^H B_{12}(s)(-\nabla V(q_0)) ds$.

    Since $B_{12}(s)=O(s)$ by Lemma \ref{ElementaryBound2}, and $\mathbb{E}\| -\nabla V(q_0) \|_2^2$ is bounded by Condition \eqref{Condition1}, one gets
    \begin{eqnarray}
        \mathbb{E} \| \tilde{q}(H)-q_H \|_2^2 &\leq& H \int_0^H \mathbb{E} \| B_{12}(s)(-\nabla V(q_0)) \|_2^2 ds \nonumber\\
        &\leq& H \int_0^H \mathbb{E} (\| B_{12}(s) \|_2 \| -\nabla V(q_0) \|_2)^2 ds \nonumber\\
        &=& H \int_0^H O(s^2) \mathbb{E}\| -\nabla V(q_0) \|_2^2 ds \nonumber\\
        &=& O(H^4)
    \end{eqnarray}

    Investigation on $p$ by applying Lemma \ref{ElementaryBound2} and Condition \ref{Condition1} gives:
    \begin{eqnarray}
        &~& \mathbb{E} \| \Omega^{-1} (\tilde{p}(H)-p_H) \|_2^2 \nonumber\\
        &=& \mathbb{E} \| \Omega^{-1} (\int_{0}^{H} B_{22}(s)ds(-\nabla V(q_0))+H\nabla V(B_{11}(H)q_0+B_{12}(H)p_0+Rq(H))) \|_2^2 \nonumber \\
        &=& \mathbb{E} \| \int_{0}^{H} \Omega^{-1} (B_{22}(s)-I)ds(-\nabla V(q_0))+H \Omega^{-1} (\nabla V(B_{11}(H)q_0+B_{12}(H)p_0 \nonumber\\
        &~& +Rq(H))-\nabla V(q_0) ) \|_2^2 \nonumber \\
        &\leq& 2 \mathbb{E} [\| \int_{0}^{H} \Omega^{-1} (B_{22}(s)-I)ds(-\nabla V(q_0)) \|_2^2 + \| H \Omega^{-1} (\nabla V(B_{11}(H)q_0+B_{12}(H)p_0 \nonumber\\
        &~& +Rq(H))-\nabla V(q_0) ) \|_2^2] \nonumber\\
        &\leq& 2 [H \int_{0}^{H} \mathbb{E} \| \Omega^{-1} (B_{22}(s)-I) (-\nabla V(q_0)) \|_2^2 ds + \mathbb{E} \| H \Omega^{-1} (\nabla V(B_{11}(H)q_0 \nonumber\\
        &~& +B_{12}(H)p_0+Rq(H))-\nabla V(q_0) ) \|_2^2] \nonumber\\
        &\leq& 2H[ \int_{0}^{H} \| \Omega^{-1} (B_{22}(s)-I) \|_2^2 ds \mathbb{E} \| (-\nabla V(q_0)) \|_2^2 + H \mathbb{E} \| \Omega^{-1} (\nabla V(B_{11}(H)q_0 \nonumber\\
        &~& +B_{12}(H)p_0+Rq(H))-\nabla V(q_0)) \|_2] \nonumber\\
        &\leq& 2H[O(H^3) + L^2 H \mathbb{E} \| \Omega^{-1} (B_{11}(H)q_0+B_{12}(H)p_0+Rq(H)-q_0) \|_2^2] \nonumber\\
        &\leq& 2H[O(H^3) + 3L^2 H (\mathbb{E} \| \Omega^{-1} (B_{11}(H)-I) q_0 \|_2^2 + \mathbb \| \Omega^{-1} B_{12}(H) p_0 \|_2^2 \nonumber\\
        &~& + \mathbb{E} \| \Omega^{-1} Rq(H) \|_2^2)] \nonumber\\
        &\leq& 2H[O(H^3) + 3L^2 H (\| \Omega^{-1} (B_{11}(H)-I) \|_2^2 \mathbb{E} \| q_0 \|_2^2 + \| B_{12}(H) \|_2^2 \mathbb{E} \| p_0 \|_2^2 \nonumber\\
        &~& + \mathbb{E} \| Rq(H) \|_2^2)] \nonumber\\
        &\leq& 2H[O(H^3) + 3L^2 H(O(H)^2 \mathbb{E} \| q_0 \|_2^2 + O(H)^2 \mathbb{E} \| p_0 \|_2^2 + O(H^3))] \nonumber\\
        &=& O(H^4)
    \end{eqnarray}

    Therefore $\mathbb{E} \| \begin{bmatrix} q_H \\ p_H \end{bmatrix} - \begin{bmatrix} \tilde{q}(H) \\ \tilde{p}(H) \end{bmatrix} \|_E^2 = O(H^4)$ independent of $\epsilon^{-1}$.
\end{proof}

\begin{Lemma}
    Consider evolutions of different local initial conditions under the bridge dynamics:
    \begin{equation}
        \left\{ \begin{array}{rcl}
            d\tilde{q}_1&=&\tilde{p}_1 dt\\
            d\tilde{p}_1&=&-\epsilon^{-1} K\tilde{q}_1 dt- \nabla V(\tilde{q}_1(0)) dt -c \tilde{p}_1 dt+\sigma dW_t\\
        \end{array} \right. \\
    \end{equation}
    \begin{equation}
        \left\{ \begin{array}{rcl}
            d\tilde{q}_2&=&\tilde{p}_2 dt\\
            d\tilde{p}_2&=&-\epsilon^{-1} K\tilde{q}_2 dt- \nabla V(\tilde{q}_2(0)) dt -c \tilde{p}_2 dt+\sigma dW_t\\
        \end{array} \right. \\
    \end{equation}
    Denote by $L$ the Lipschitz coefficient of $\nabla V(\cdot)$ (i.e. $\| \nabla V(a)-\nabla V(b) \|_2 \leq L \| a-b \|_2$), then almost surely
    \begin{equation}
        \| \begin{bmatrix} \tilde{q}_1(H)-\tilde{q}_2(H) \\ \tilde{p}_1(H)-\tilde{p}_2(H) \end{bmatrix} \|_E
        \leq (1+HL) \| \begin{bmatrix} \tilde{q}_1(0)-\tilde{q}_2(0) \\ \tilde{p}_1(0)-\tilde{p}_2(0) \end{bmatrix} \|_E
    \end{equation}
    \label{lemma3}
\end{Lemma}

\begin{proof}
    Write out the solution to the bridge dynamics in integral form:
    \begin{eqnarray}
        \begin{bmatrix} \tilde{q}_1(H) \\ \tilde{p}_1(H) \end{bmatrix} &=& B(H) \begin{bmatrix} \tilde{q}_1(0) \\ \tilde{p}_1(0) \end{bmatrix} +\int_0^H B(H-s) \begin{bmatrix} 0 \\ -\nabla V(\tilde{q}_1(0)) \end{bmatrix} ds + \int_0^H B(H-s) \Sigma dW_s \nonumber\\
        \begin{bmatrix} \tilde{q}_2(H) \\ \tilde{p}_2(H) \end{bmatrix} &=& B(H) \begin{bmatrix} \tilde{q}_2(0) \\ \tilde{p}_2(0) \end{bmatrix} +\int_0^H B(H-s) \begin{bmatrix} 0 \\ -\nabla V(\tilde{q}_2(0)) \end{bmatrix} ds + \int_0^H B(H-s) \Sigma dW_s \nonumber\\
    \end{eqnarray}

    Hence almost surely
    \begin{eqnarray}
        &~& \| \begin{bmatrix} \tilde{q}_1(H)-\tilde{q}_2(H) \\ \tilde{p}_1(H)-\tilde{p}_2(H) \end{bmatrix} \|_E \nonumber\\
        &\leq& \| B(H) \begin{bmatrix} \tilde{q}_1(0)-\tilde{q}_2(0) \\ \tilde{p}_1(0)-\tilde{p}_2(0) \end{bmatrix} \|_E +\int_0^H \| B(H-s) \begin{bmatrix} 0 \\ \nabla V(\tilde{q}_2(0))-\nabla V(\tilde{q}_1(0)) \end{bmatrix} \|_E ds \nonumber\\
        &\leq& \| \begin{bmatrix} \tilde{q}_1(0)-\tilde{q}_2(0) \\ \tilde{p}_1(0)-\tilde{p}_2(0) \end{bmatrix} \|_E +H \| \begin{bmatrix} 0 \\ \nabla V(\tilde{q}_2(0))-\nabla V(\tilde{q}_1(0)) \end{bmatrix} \|_E \nonumber\\
        &\leq& \| \begin{bmatrix} \tilde{q}_1(0)-\tilde{q}_2(0) \\ \tilde{p}_1(0)-\tilde{p}_2(0) \end{bmatrix} \|_E + HL \| \begin{bmatrix} 0 \\ \tilde{q}_2(0)-\tilde{q}_1(0) \end{bmatrix} \|_E \nonumber\\
        &\leq& \| \begin{bmatrix} \tilde{q}_1(0)-\tilde{q}_2(0) \\ \tilde{p}_1(0)-\tilde{p}_2(0) \end{bmatrix} \|_E + HL \| \begin{bmatrix} \tilde{q}_2(0)-\tilde{q}_1(0) \\ 0 \end{bmatrix} \|_E \nonumber\\
        &\leq& (1+HL) \| \begin{bmatrix} \tilde{q}_1(0)-\tilde{q}_2(0) \\ \tilde{p}_1(0)-\tilde{p}_2(0) \end{bmatrix} \|_E
    \end{eqnarray}
\end{proof}

\begin{Remark}
    If the traditional method of investigating the Lipschitz coefficient of the vector field is employed to evolve the separation of local initial conditions, $\epsilon^{-1}$ will exhibit in the bound of separation. Instead we only looked at the soft part of the vector field and whence obtained a uniform bound.
\end{Remark}

\noindent \textbf{Theorem \ref{UniformConvergence} (global error bound in energy norm).}
\begin{proof} ~

    \xymatrix{
    x(NH)
    \ar@{~}[rr]^{O(H^{3/2})}_{\text{root mean square}}
    &&
    \tilde\alpha
    \ar@{~}[rrr]^{e_{N-1}(1+HL)}_{\text{almost surely}}
    &&&
    \tilde\beta
    \ar@{~}[rr]^{O(H^2)}_{\text{root mean square}}
    &&
    x_{NH}
    \\
    \\
    \\
    &&
    x((N-1)H)
    \ar@{~}[rrr]_{e_{N-1}}
    \ar[uuull]^{\text{original dynamics}}
    \ar[uuu]^{\text{bridge}}_{\text{dynamics}}
    &&&
    x_{(N-1)H}
    \ar[uuu]^{\text{bridge}}_{\text{dynamics}}
    \ar[uuurr]_{\text{1st order SIM}}
    }

    \begin{center} \textbf{Atlas of error propagation} \end{center}
    \medskip

    Let $e_N = (\mathbb{E} \| x(NH)-x_{NH} \|_E^2)^{1/2}$. Let $\tilde{\alpha}$ and $\tilde{\beta}$ be respectively the evolution of the real solution $x((N-1)H)$ and the numerical solution $x_{(N-1)H}$ by time $H$ under the bridge dynamics \eqref{BridgeDynamics2}.

    Then by Lemma \ref{lemma1} and \ref{lemma2}, there exist constants $C_1$ and $C_2$ independent of $\epsilon^{-1}$ such that

    \begin{eqnarray}
        (\mathbb{E} \| X(NH)-\tilde{\alpha} \|_E^2)^{1/2} &\leq& C_1 H^{3/2} \nonumber\\
        (\mathbb{E} \| \tilde{\beta}-X_{Nh} \|_E^2)^{1/2} &\leq& C_2 H^2
    \end{eqnarray}

    Also since $\| \tilde{\alpha}-\tilde{\beta} \|_E \leq (1+HL) \| x((N-1)h)-x_{(N-1)h} \|_E \text{ almost surely}$ (Lemma \ref{lemma3}), we have:
    \begin{equation}
        (\mathbb{E} \| \tilde{\alpha}-\tilde{\beta} \|_E^2)^{1/2} \leq (1+HL) e_{N-1}
    \end{equation}

    All in all,
    \begin{eqnarray}
        e_N &\leq& (\mathbb{E} \| x(NH)-\tilde{\alpha} \|_E^2)^{1/2} + (\mathbb{E} \| \tilde{\alpha}-\tilde{\beta} \|_E^2)^{1/2} + (\mathbb{E} \| \tilde{\beta}-X_{NH}) \|_E^2)^{1/2} \nonumber\\
        &\leq& (1+HL) e_{N-1} + (C_1+C_2)H^{3/2} \nonumber\\
        &=& (1+HL)^N e_0 + (C_1+C_2)H^{3/2} \frac{(1+HL)^N-1}{(1+HL)-1} \nonumber\\
        &\leq& (C_1+C_2)H^{1/2} \frac{e^{NHL}-1}{L} = \frac{(C_1+C_2)(e^{TL}-1)}{L}H^{1/2}
    \end{eqnarray}

    Therefore letting $C=\frac{(C_1+C_2)(e^{TL}-1)}{L}$ we have
    \begin{align}
        (\mathbb{E} \| q(T)-q_T \|_2^2)^{1/2} &\leq e_N \leq C H^{1/2} \\
        (\mathbb{E} \| p(T)-p_T \|_2^2)^{1/2} &\leq \epsilon^{-1/2} \| \sqrt{K} \|_2 e_N \leq \epsilon^{-1/2} \| \sqrt{K} \|_2 C H^{1/2}
    \end{align}
\end{proof}

\bibliographystyle{siam}
\bibliography{molei14}

\end{document}